\documentclass[12pt]{amsart}
\usepackage[leqno]{amsmath}
\usepackage{amssymb}
\usepackage{enumerate}
\usepackage[all]{xy}
\usepackage{appendix}
\newtheorem{theorem}{Theorem}
\newtheorem{lemma}[theorem]{Lemma}
\newtheorem{proposition}[theorem]{Proposition}

\newcommand{\E}{\mathbb{E}}

\newcommand{\N}{\mathbb{N}}
\newcommand{\Z}{\mathbb{Z}}
\newcommand{\p}{\mathbb{P}}

\newcommand{\R}{\mathbb{R}}

\newcommand{\CA}{\mathcal {A}}

\newcommand{\CC}{\mathcal {C}}

\newcommand{\CE}{\mathcal {E}}

\newcommand{\CM}{\mathcal {M}}

\newcommand{\CQ}{\mathcal {Q}}

\newcommand{\CH}{\mathcal {H}}

\newcommand{\supp}{\text{supp }}

\input xy
\xyoption{all}

\begin{document}

\title{Quasi-stationary Random Overlap Structures and the continuous Cascades}
\author{Jason Miller}
\email{jmiller@math.stanford.edu}
\address{
\noindent Stanford University,
Department of Mathematics,
Building 380,
Stanford, CA 94305
}
\thanks{Research supported in part by NSF grants DMS-0406042 and DMS-0806211.}
\date{\today}
\maketitle

%%TODO:
%%   Add in the argument that the overlaps cannot be negative - Wednesday
%%   Update the references - Wednesday

%%   Fix the argument re CLT

%% Thursday, Friday, Saturday, Sunday - Check, revise
%% Sunday: Submit

\begin{abstract}
A random overlap structure (ROSt) is a measure on pairs $(X,Q)$ where $X$ is a locally finite sequence in $\R$ with a maximum and $Q$ a symmetric positive semidefinite matrix of overlaps intrinsic to the particles $X$.  Such a measure is said to be quasi-stationary provided that the joint law of the gaps of $X$ and overlaps $Q$ is stable under a stochastic evolution driven by a Gaussian sequence with covariance $Q$.  Aizenman et al. show in \cite{AS203} that quasi-stationary ROSts serve as an important computational tool in the study of the Sherrington-Kirkpatrick (SK) spin-glass model from the perspective of cavity dynamics and the related ROSt variational principle for its free energy.  In this framework, the Parisi solution is reflected in the ansatz that the overlap matrix exhibit a certain hierarchical structure.  Aizenman et al. pose the question in \cite{AS207} of whether the ansatz could be explained by showing that the only ROSts that are quasi-stationary in a robust sense are given by a special class of hierarchical ROSts known as both the Ruelle Probability Cascades as well the GREM.  Arguin and Aizenman give an affirmative answer in \cite{AA07} and \cite{A_GG07} in the special case that the set of values $S_Q$ taken on by the entries of $Q$ is finite.  We prove that this result holds even when $|S_Q| = \infty$ provided that $Q$ satisfies the technical condition that $\overline{S_Q}$ has no limit points from below.  This is a relevant step towards understanding the ground states of the SK model, as they satisfy $|S_Q| = \infty$.
\end{abstract}

\section{Introduction}
A random overlap structure (ROSt) is a measure on pairs $(X,Q)$ where $X = (X_1 \geq X_2 \geq \cdots)$ is a locally finite sequence in $\R$ with a maximum and $Q = (q_{ij})$ a symmetric positive semidefinite $\N \times \N$ matrix of overlaps with $q_{ii} = 1$ for all $i$.  The overlaps are intrinsic to the particles $X$ and serve to quantify the degree to which they are related.  A ROSt is said to be quasi-stationary if the joint law of the gaps $(X_i - X_{i+1})$ and overlaps $Q$ is stable under the stochastic evolution
\begin{eqnarray}
\label{eqn::qs_not_norm}
X_i \mapsto X_{\pi(i)} + \psi(\kappa_{\pi(i)}),\ Q = (q_{ij}) \mapsto Q^\pi := (q_{\pi(i) \pi(j)}),
\end{eqnarray}
where conditional on $Q$, $(\kappa_i)$ is a Gaussian sequence with covariance $Q$ and $\pi$ is a permutation that restores the descending order of $X$.  Here, we assume $\psi \in \CC$ where $\CC$ is the set of all functions in $C^2(\R)$ with $\psi'(0) \neq 0$ that have the property $\E \exp(\lambda \psi(Z)) < \infty$ for all $\lambda \in \R$ and $Z$ a standard Gaussian.  Note that the evolution is defined so that the overlap of any two given particles is the same both before and after one step of the evolution.
From the cavity perspective of the Sherrington-Kirkpatrick (SK) model (see \cite{AS207}), it is natural to restrict our attention to those ROSts satisfying
\begin{equation}
\label{eqn::exponential_sum}
 \sum_i e^{\beta X_i} < \infty \text{ for some } \beta > 0.
\end{equation}
If we let $\xi$ be the normalized process $\xi_i = e^{\beta X_i}/\sum_j e^{\beta X_j}$ then the quasi-stationarity of (the law of) $(X,Q)$ translates into the stability of (the law of) $(\xi,Q)$ under the stochastic mapping
\begin{eqnarray}\label{eqn::qs}(\xi,Q) \mapsto \left( \left(\frac{ \xi_i e^{\psi(\kappa_i)}}{\sum_j \xi_j e^{\psi(\kappa_j)}} \right)_{\downarrow}, Q^\pi \right) ,\end{eqnarray}
where $\downarrow$ indicates a reordering and $\pi$ is again the order restoring permutation.  We break ties by saying that particles which have equal weight in the evolved structure are ordered according to their position in the original structure.  In the case that $\xi_i \equiv 0$ for all $i$, we take the evolved weights also to be identically zero.  Let $d_{ij} = \delta_{ij} -q_{ij}$.  We say that $Q$ is ultrametric if $d_{ij} \leq \max(d_{ik},d_{kj})$ for all $i,j,k$.  Note that ultrametricity implies that the relation $i \sim_q j$ if and only if $q_{ij} \geq q$ is an equivalence.

Motivated by the cavity picture and the Parisi ansatz (see \cite{PA80}, \cite{TAL03}, \cite{TAL06}) for the SK model it is of considerable interest to characterize those structures that are quasi-stationary and to understand the relationship between quasi-stationarity and the ultrametricity of the overlap matrix.  The first step in this direction was taken by Ruzmaikina and Aizenman in \cite{RA05} where they show that, in the special case $Q$ is the identity matrix so that the Gaussian sequence $(\kappa_i)$ is iid, the quasi-stationary laws (without the normalization) are given by mixtures of Poisson processes with exponential intensities, i.e. of the form $ye^{-ys} ds$ with $y > 0$.  The Ruzmaikina-Aizenman theorem holds under the technical condition that there exists constants $A, r > 0$ such that
\begin{equation}
\label{eqn::expected_bound}
\E|\{ n : X_1 - X_n \leq d\}| \leq A e^{rd},
\end{equation}
where $|\cdot|$ denotes the cardinality of a set.  Recall that $\xi$ is said to follow the Poisson-Dirichlet distribution with parameter $x \in (0,1)$, denoted by $\xi \sim PD(x,0)$, if $(\xi_i) \stackrel{d}{=} (X_i / \sum_j X_j)$, where $X = (X_1 \geq X_2 \geq \cdots \geq 0)$ is a Poisson process with intensity $xs^{-x-1}ds$.  Note that densities of the latter form arise from the former via an exponential change of variables.
In \cite{A07}, Arguin shows that the Ruzmaikina-Aizenman theorem is also true when the hypothesis \eqref{eqn::expected_bound} is replaced by \eqref{eqn::exponential_sum} which implies that the normalized quasi-stationary ROSts $(\xi,Q)$ with $Q$ the identity matrix are exhausted by the mixtures of $PD(x,0)$ variables.  In this article we restrict our attention to quasi-stationarity in the normalized case.

For $r \in \N$, let $Q^{*r} = (q_{ij}^r)$ denote the $r$th Schur power of $Q$, i.e. entrywise product of $Q$ with itself $r$ times.  If $Q$ is symmetric and positive semidefinite then $Q^{*r}$ is as well \cite{HJ85}.  Fix $\psi \in \CC$ and for $r \in \N$, $\lambda \in \R$, let $\Phi_{r,\lambda}$ be the stochastic map,
\[ \Phi_{r,\lambda}(\xi,Q) := \left(\left( \frac{\xi_i e^{\lambda \psi(\kappa_i)}}{\sum_j \xi_j e^{\lambda \psi(\kappa_j)}} \right)_\downarrow, Q^\pi \right),\]
 where the Gaussian sequence $(\kappa_i)$ is generated with covariance $Q^{*r}$.  We say that a ROSt is $(r,\lambda)$-quasi-stationary for $\psi$ if its law is stable $\Phi_{r,\lambda}$.  A ROSt is said to be robustly quasi-stationary (RQS) provided that it is $(r,\lambda)$-quasi-stationary for all $r \in \N$ and $\lambda > 0$.  Let $S_Q = \{q_{ij} : i,j \in \N, i \neq j\}$ denote the state space of $Q$ and $S_Q(i) = \{q_{ij} : j \neq i, j \in \N\}$.  The overlap matrix $Q$ is said to be indecomposable provided that $S_Q(i) = S_Q$ for every $i$; this condition represents a sort of self-similarity of the particles which is natural in the context of the SK model.  In \cite{AA07}, Arguin and Aizenman show that if $(\xi,Q)$ is RQS, ergodic for the correlated evolution with $S_Q \subseteq (-1,1)$, $|S_Q| < \infty$, and $Q$ indecomposable then $(\xi,Q)$ is given by a finite level Ruelle Probability Cascade (RPC), a special type of ROSt with ultrametric overlaps arising from a hierarchy of Poisson processes which we describe in the next section.  In the case that $Q$ is not indecomposable, there is a natural equivalence between particles where $\xi_i \sim \xi_j$ if and only if $S_Q(i) = S_Q(j)$.  As the set $S_Q(i)$ is intrinsic to $\xi_i$, this equivalence is invariant under the evolution so that the equivalence classes are themselves RQS, ergodic, and indecomposable and hence given by RPCs \cite{AA07}.  In \cite{A_GG07}, Arguin shows that the assumption of indecomposability can be removed by arguing that the equivalence classes have the same state space and hence must be equal.  By convexity, the results of \cite{AA07} and \cite{A_GG07} imply that the set of RQS ROSts with finite state space are given by superpositions of finite level RPCs.  Our main theorem is that this result holds even if $|S_Q| = \infty$ provided that $S_Q$ satisfies the technical condition that $\overline{S_Q}$ does not have limit points from below:
\begin{theorem}
\label{thm::main_theorem}
If $(\xi,Q)$ is a robustly quasi-stationary random overlap structure with $S_Q \subseteq (-1,1)$ and ergodic for the evolution \eqref{eqn::qs} such that if $q \in \overline{S_Q}$ then $\sup\{ p \in \overline{S_Q}, p < q\} < q$, then $(\xi,Q)$ is given by a continuous Ruelle Probability Cascade.  In particular, all such laws are ultrametric and satisfy $S_Q \subseteq [0,1)$.
\end{theorem}
\noindent The condition of robustness is necessary to single out the cascades.  At the end of the next section, we will show that for each fixed $R \in \N$ there exists a ROSt that is $(r,\lambda)$-quasi-stationary with respect to $\psi(x) = x$ for all $\lambda > 0$ and all $r \leq R$ that is not an RPC.  The relaxation of the restriction that $S_Q$ is finite is a relevant step towards the understanding of the ground states of the SK model as in this case $|S_Q| = \infty$.  Note that the assumption that $\overline{S_Q}$ does not have limit points from below allows for $S_Q$ to be rather complicated.  For example, the number of cluster points of $S_Q$ could be infinite.

The proof in \cite{AA07} is by induction on $|S_Q|$ and has the following structure.  First, assume that $\psi(x) = \lambda x$ is linear.  Taking a limit as $r \to \infty$ the non-diagonal entries of $Q^{*r}$ tend to zero from which it follows that RQS implies quasi-stationarity under the free evolution, the evolution driven by an iid Gaussian sequence $(\kappa_i)$.  The latter gives that $\xi$ is independent of $Q$ and, by \cite{A07}, $\xi$ follows a $PD(x,0)$ distribution.  By looking at the past increments, it can be deduced that $Q$ is weakly exchangeable.  Hence by a version of de Finetti's theorem for weakly exchangeable covariance matrices proved in \cite{DS82}, there exists a random measure $\nu$ on a separable Hilbert space $\CH$ so that conditional on $\nu$ the distribution of $Q=(q_{ij})$ is given by $q_{ij} \stackrel{d}{=} (\phi_i,\phi_j) + \delta_{ij} (1-\|\phi_i\|^2)$ where $(\phi_i)$ is a $\nu$-iid sequence; we refer to $\nu$ as the \emph{directing measure} of $(\xi,Q)$.   That $S_Q$ is finite and indecomposable is used to conclude that $\nu$ admits the representation $\nu = \sum_j \eta_j \delta_{\psi_j}$ where $\| \psi_j\|^2 = \sup S_Q < 1$ is constant for all $j$.  Viewing $(\xi,Q)$ as a marked Poisson-Dirichlet process with marking measure $\nu$, RQS implies that the ROSt $(\eta,P)$, $P =(p_{ij})$, $p_{ij} = (\psi_i,\psi_j)/\sup S_Q$ is itself RQS but with $|S_P| = |S_Q|-1$.  The induction step follows as $(\xi,Q)$ can be reconstructed from $(\eta,P)$ in exactly the same way that a $k$-level RPC can be constructed from a $k-1$ level RPC (see Proposition \ref{prop::rpc_construction}).  The case that $\psi$ is non-linear can be reduced to the linear case via an argument with the central limit theorem.

A new proof is required in the infinite case for two reasons.  First, the argument that the directing measure is countably supported on a non-random sphere crucially depends on the indecomposability of $Q$.  It turns out that this holds more generally, but requires a different argument.  Second, the induction fails in the infinite case due to the possibility of many cluster points in $S_Q$.  Our approach is as follows. We use a continuity argument in Section \ref{sec::closure} to show that the closure of the set of finite level RPCs is the set of continuous RPCs provided we equip the space of ROSts with an appropriate topology.  Next, we show in Section \ref{sec::approx} that ergodic RQS ROSts with state space that has no limit points from below can be approximated arbitrarily well by such ROSts with indecomposable finite state space.  This step is non-trivial since it requires constructing an approximation that remains a fixed point of an infinite collection of maps.

\section{Ruelle Probability Cascades}

In this section we recall the construction of the finite level RPCs given by Ruelle in \cite{RU87}, following closely that given in \cite{AA07}.  Fix $k \in \N$ and $2k$ parameters
\[ 0 = q_0 \leq q_1 < q_2 < \cdots <  q_{k+1} =1,\]
and
\[ 0 = x_0 <  x_1 < x_2 < \cdots < x_{k+1} = 1.\]
Define a right-continuous function on $[0,1]$,
\[ x(q) =  \sum_{l=0}^k x_l 1_{[q_l,q_{l+1})}(q)\]
and let $q(x)$ be the right continuous inverse of $x(q)$:
\[ q(x) = \inf\{ t : x(t) > x\} = \sum_{l=1}^{k+1} q_l 1_{[x_{l-1},x_l)}(x).\]
One often specifies $x(q)$ as the order parameter of an RPC, but for us it will be more convenient to work with $q(x)$.  This perspective is without loss of generality as the association of $x(q)$ with $q(x)$ is injective.  Let $\alpha = (\alpha_1,\ldots,\alpha_k) \in \N^k$.  Denote by $\alpha(l) = (\alpha_1,\ldots,\alpha_l)$ the truncation of $\alpha$ up to its $l$th component; set $\alpha(0) = 0$.  For each $\alpha(l)$, $0 \leq l \leq k-1$, let $\eta^{\alpha(l)}$ be an independent Poisson process with intensity $x_{l+1} s^{-x_{l+1}-1} ds$.  We denote by $\eta_j^{\alpha(l)}$ the $j$th largest element of $\eta^{\alpha(l)}$.  Now define the point process,
\[ \eta = (\eta_\alpha : \alpha \in \N^k) = (\eta_{\alpha_1}^0 \eta_{\alpha_2}^{\alpha(1)} \cdots \eta_{\alpha_k}^{\alpha(k-1)} : \alpha \in \N^k).\]
As each $\eta^{\alpha(l)}$ is summable it follows that $\eta$ is summable and can be ordered.  Let
\[ N = \sum_{\alpha} \eta_\alpha \text{ and } \xi = \left( \frac{1}{N} \eta_\alpha \right)_{\downarrow}\]
be the normalization of the ordering of $\eta$ to have unit sum.  Fix $i,j$.  Then $\xi_i = N^{-1} \eta_\alpha$ and $\xi_j = N^{-1} \eta_\beta$ for some $\alpha,\beta \in \N^k$.  Let $\sim_l$ be the equivalence relation on $\N$ given by $i \sim_l j$ if and only if the truncations $\alpha(l)$ and $\beta(l)$ agree.  Let $l(i,j)$ be the largest $l$ such that $i \sim_l j$, set 
\[q_{ij} = q_{l(i,j)+1} = q(x_{l(i,j)}),\]
and let $Q = (q_{ij})$.  Then $(\xi,Q)$ defines a ROSt, which we refer to as a $k$-level RPC; we will also refer to $(\xi,Q)$ as a finite level RPC if we do not wish to specify the number of levels.  We can think of the association of the parameter $q(x)$ with $(\xi,Q)$ given by the RPC with order parameter $q(x)$ as a map from the set of finite valued non-decreasing right-continuous step functions on $[0,1]$ into the space of ROSts.  In the sequel we will prove that if we equip these spaces with appropriate (and natural) topologies, this map is a homeomorphism when the parameter $x_k$ is fixed but the number of levels $k$ is allowed to vary.

Let 
\[u_1 = \log\left(\frac{x_k}{x_1}\right), u_2 = \log\left( \frac{x_k}{x_2}\right),\ldots,u_k = \log\left(\frac{x_k}{x_k}\right) = 0\]
and denote by $\Gamma_{u_j}$ the equivalence classes of $\sim_j$ in $\N$.  Then we can think of $(\Gamma_{u_j} : j=k,\ldots,1)$ as a process on $\CE(\N)$, the set of equivalence relations on $\N$.  It turns out that $(\Gamma_{u_j})$ is the discrete skeleton of a continuous time Markov process  $(\Gamma_u : u \geq 0)$ with initial distribution the equivalence given by equality and semigroup $R_u(\Gamma,d\Gamma')$, $u \geq 0$, given as follows (Proposition 1.1, Theorem 1.2, Theorem 2.2 \cite{BS98}).  When $u = 0$, $R_u(\Gamma,d\Gamma')$ is the $\delta$ mass supported at $\Gamma$.  When $u > 0$, construct a random variable $\Gamma' $ taking values in $\CE(\N)$ by first generating a $PD(x,0)$ variable $\eta$ with $x = e^{-u}$, then picking $\eta$ (thought of as a probability on $\N$) iid random integers $y_C$ indexed by $C \in \Gamma$, and letting $\Gamma'$ be the set of equivalence classes,
\[ C_j' = \bigcup_{y_C = j} C.\]
Then $R_u(\Gamma,d\Gamma')$ is the law of $\Gamma'$.  The process $(\Gamma_u)$ is called the Bolthausen-Sznitman coalescent and can be used to give a concrete specification of Ruelle's direct limit construction of the continuous RPCs.  If $q \colon [0,1] \to [0,1]$ is a non-decreasing right-continuous function such that $q(\zeta) = 1$ for $\zeta \in (0,1)$ we associate with $(\Gamma_u)$ a ROSt $(\xi,Q)$ where $\xi \sim PD(\zeta,0)$ is independent of $(\Gamma_u)$ and $Q = (q_{ij})$ where $q_{ij} = q(\tau_{ij})$, 
\[  \tau_{ij} =  \inf\{ u > 0 : i,j \text{ are in the same } \Gamma_u \text{ equivalence class}\}.\]
We refer to $(\xi,Q)$ as a continuous RPC with parameters $q,\zeta$.  In the special case that $q$ is a step function with $\zeta = \inf\{ x : q(x) = 1\}$ then $(\xi,Q)$ has the law of the finite level RPC parameterized by $q$.

An immediate consequence of this is the following alternative construction of the finite level RPCs originally due to Bolthausen and Sznitman in \cite{BS98}, through reformulated as follows as Theorem 2.5 of \cite{AA07}.
\begin{proposition}
\label{prop::rpc_construction}
The distribution of the RPC $(\xi,Q)$ with parameters $0 = q_0 \leq q_1 < \cdots  < q_k <  q_{k+1} = 1$, $0 = x_0 <  x_1 < \cdots < x_k < x_{k+1} = 1$ satisfies the following:
\begin{enumerate}
 \item $\xi$ and $Q$ are independent,
 \item $\xi \sim PD(x_k,0)$, and
 \item Suppose that $(\xi',Q') = (q_{ij}')$ is independent of $\xi$ and distributed as a $k-1$ level RPC with parameters $q_1/q_k < \cdots < q_{k-1}/q_k$ and $x_1/x_k < \cdots < x_{k-1}/x_k$.  Then, conditional on $(\xi',Q')$, $Q \stackrel{d}{=} (q_{ij})$ where $q_{ij} = q_k q_{i^* j^*}'$ with $i^*$ denoting the $i$th element of a sequence $(i^*)$ of $\xi'$ (thought of as a probability on $\N$) iid random integers.
\end{enumerate}
\end{proposition}

Proposition \ref{prop::rpc_construction} is used in the proof of Theorem 2.9 of \cite{AA07} to show that the finite level RPCs are RQS.  We will sketch the argument in the case that $\psi(x) = \lambda x$ is linear, the general case being roughly the same but with more complicated notation.  First, note that if $(\xi,Q)$ is a $k$-level RPC with order parameter $q(x)$ then $(\xi,Q^{*r})$ is also a $k$-level RPC but with order parameter $q^r(x)$.  Hence that $(\xi,Q)$ is RQS will follow if we can show that $(\xi,Q)$ is quasi-stationary for the correlated evolution.  When $k=1$ we can decompose the Gaussian increments as $\kappa_i = \kappa^c + \kappa_i^f$ where $\kappa^c$ is the common shift and $(\kappa_i^f)$ is an iid sequence of Gaussian random variables with variance $(1-q_1^2)$.  The common part $\kappa^c$ gets canceled in the normalization \eqref{eqn::qs} and hence in this case the correlated evolution is the same as the free evolution.  As $\xi \sim PD(x_1,0)$, the law of $\xi$ is invariant; $Q$ is trivially invariant as its non-diagonal entries are constant.  Now suppose $(\xi,Q)$ is a $k \geq 2$ level cascade and let $(\xi',Q')$ be a $k-1$ level cascade as in Proposition \ref{prop::rpc_construction}.  Conditional on $(\xi',Q')$, $(\xi,Q)$ can be viewed as a marked Poisson-Dirichlet variable with the mark of $\xi_i$ being its ``parent'' in $(\xi',Q')$.  The Gaussian sequence of increments $(\kappa_i)$ can again be decomposed as $\kappa_i = \kappa_i^c + \kappa_i^f$ where $\kappa_i^f$ are iid with variance $(1-q_{k}^2)$.  Conditional on the realization of $\kappa_i^c$, the evolution is that of a marked $PD(x_k,0)$ variable with mark dependent increments.  One can show using the probability mass generating functional (Lemma 2.7 of \cite{AA07}, see also Proposition A.2 of \cite{BS98}) that the resulting process is again a marked $PD(x_k,0)$ variable but with a new marking measure.  It turns out that in this case the evolution of the marking measure exactly corresponds to the correlated evolution of the $k-1$ level cascade $(\xi',Q')$, which we know by induction to be quasi-stationary.  As the set of RQS laws is closed (Proposition \ref{prop::rqs_closed}), it follows from the continuity argument given in the next section that the continuous RPCs are also RQS.

We will now show that the hypothesis of robustness is necessary to single out the cascades using a construction due to Arguin \cite{A08}.  Suppose that $(\xi,Q)$ is a finite level cascade with parameter $q(x)$ and $\xi \sim PD(\zeta,0)$, $\zeta \in (0,1)$.  The proof of Theorem 2.9 of \cite{AA07} implies that if $q_1 = 0$ then for each $r \in \N$ and $\lambda > 0$ there exists a deterministic quantity $v(r,\lambda)$ such that
\[ \left( \left( \xi_n e^{\lambda \kappa_n} \right)_\downarrow, Q^\pi\right) \stackrel{d}{=} \big( (e^{v(r,\lambda)} \xi_n), Q \big),\]
where the $(\kappa_n)$ are Gaussian with covariance $Q^{*r}$; $v(r,\lambda)$ is referred to as the $(r,\lambda)$-crowd velocity of $(\xi,Q)$.  In other words, the renormalization constant in \eqref{eqn::qs} is a deterministic function of $r$, $\lambda$, and the parameters of the RPC.  In fact, one can compute $v(r,\lambda)$ explicitly (Theorem 5.4, \cite{AS207}),
\[ v(r,\lambda) = \frac{\lambda^2}{2} \int_0^1 (1-q^r) dx(q).\]
This implies that if $(\xi^1,Q^1)$ and $(\xi^2,Q^2)$ are independent RPCs with equal $(r,1)$-crowd velocity, the ROSt $(\xi,Q)$ constructed by taking the union of the particles in $\xi^1,\xi^2$ with overlaps between particles of $\xi^i$ given by $Q^i$ and between $\xi^1,\xi^2$ given by $0$ is $(r,\lambda)$-quasi-stationary for all $\lambda > 0$.  Such ROSts are easily seen not to be RPCs in general.  Furthermore, fixing $R \in \N$ and then choosing distinct RPCs $(\xi^1,Q^1),(\xi^2,Q^2)$ so that they have the same $(r,1)$-velocity for all $r \leq R$ one can construct examples of ROSts that are $(r,\lambda)$-quasi-stationary for all $r \leq R$ and $\lambda > 0$ that are not RPCs.  This technique can be extended to build much more elaborate examples of quasi-stationary ROSts using a variation of Proposition \ref{prop::rpc_construction} and RPCs as the basic building blocks.  Nevertheless, such ROSts still have ultrametric overlaps and it remains an important question to determine if quasi-stationarity alone implies ultrametricity.

\section{Preliminaries}
\label{sec::closure}

The goal of this section is to prove that the set of RQS ROSts is closed and that the closure of the set of finite level RPCs is given by the set of continuous RPCs.  In order to do this, we need to specify a topology (same as is used in \cite{AA07}). Let $P_m$ be the set of mass partitions $\{\xi_1 \geq \xi_2 \geq \cdots \geq 0 : \sum_{i=1}^\infty \xi_i \leq 1\}$ and $\CQ$ the set of symmetric positive semidefinite $\N \times \N$ matrices $Q = (q_{ij})$ with $q_{ii} = 1$ for all $i$.  We equip $P_m$ with the metric induced by the $\ell^\infty$ norm and $\CQ$ with the product topology.  The latter can be metrized with the distance,
\[ d(Q,P) = \sum_{i,j=1}^\infty 2^{-i-j} |q_{ij} - p_{ij}|.\]
Note that both of these topologies are compact.  Denote by $\Omega_{os} = P_m \times \CQ$ the set of overlap structures and $\CM_1(\Omega_{os})$ the set of ROSts, i.e. the Borel regular probability measures on $\Omega_{os}$.  We equip $\CM_1(\Omega_{os})$ with the weak topology.  Observe that $\Phi_r = \Phi_{r,1}$ is \emph{not} continuous with respect to this topology.  For example, if one considers the sequence of ROSts $(\xi^n,Q^n)$ where $\xi^n$ consists of a single particle at $\frac{1}{n}$ and $Q^n$ is the identity matrix then $(\xi^n,Q^n) \to (\xi,Q)$ which is the ROSt where all of the particles have zero mass while for all $n$, $\Phi_r(\xi^n,Q^n)$ consists of a single particle at $1$.  Nevertheless, the set of RQS ROSts is closed in the $\Omega_{os}$ topology.  We remark that if one changes the topology of $P_m$ to that induced by the $\ell^1$ norm then $\Phi_r$ is in fact continuous but this topology is not compact so that we do not have the desired ergodic decomposition.

\begin{proposition}
\label{prop::rqs_closed}
The set of RQS laws is closed in $\CM_1(\Omega_{os})$.
\end{proposition}
\begin{proof}
Observe that $(\xi,Q) \mapsto \Phi_r(\xi,Q)$ is weakly continuous in $Q$.  Suppose that $(\xi^n,Q^n)$ is a sequence of RQS ROSts converging in the topology of $\CM_1(\Omega_{os})$ to $(\xi,Q)$.  By Theorem 1.9 of \cite{AA07} we know that for each $n$ there exists a parameter $X_n \in (0,1)$ measurable with respect to $\sigma(\xi^n)$ such that conditional on $X_n$, $\xi^n \sim PD(X_n,0)$.  By passing to a subsequence we may assume without loss of generality that $X_n$ converges weakly to $X \in [0,1]$.  A simple computation with the mass generating functional shows that conditional on $X$, $\xi \sim PD(X,0)$.  Furthermore, as $\xi^n$ is independent of $Q^n$ for all $n$ it follows that $\xi$ is independent of $Q$.  Proposition \ref{prop::rescale_direct} implies that if $(\eta,P)$ is RQS and $\xi$ is a mixture of Poisson-Dirichlet distributions independent of $P$ then $(\xi,P)$ is RQS.  We can embed $(\xi^n,Q^n)$ and $(\xi,Q)$ into a common probability space such that $\xi$ is independent of $Q^n$ for all $n$ so that in particular $(\xi,Q^n)$ is RQS.  Combining everything,
\begin{align*}
   (\xi,Q) \stackrel{d}{=} \lim_n (\xi,Q^n) \stackrel{d}{=} \lim_n \Phi_r(\xi,Q^n) \stackrel{d}{=} \Phi_r(\xi, \lim_n Q^n),
\end{align*}
from which the proposition follows.
\end{proof}

Fix $\zeta \in (0,1)$ and let $\overline{X} = \overline{X}(\zeta)$ be the set of functions $q$ on $[0,1]$ that are right-continuous, non-decreasing, and $q(x) = 1$ for all $x \geq \zeta$.  Let $X$ be the set of $q \in \overline{X}$ with finite range.  Equip $\overline{X}$ with the topology induced by the $L^1([0,1])$ norm; observe that the closure of $X$ is $\overline{X}$.  The association of the parameter $q \in X$,
\[ q(x) = \sum_{l=1}^{k+1} q_l 1_{[x_{l-1},x_l)},\ x_k = \zeta,\]
with a finite level RPC induces a map $T \colon X \to \CM_1(\Omega_{os})$ that is trivially invertible.  Let $Y = Y(\zeta) = T(X)$ be the set of finite level RPCs $(\xi,Q)$ in $\CM_1(\Omega_{os})$ with $\xi \sim PD(\zeta,0)$.  Denote by $\overline{Y}$ its closure.
\begin{proposition}
\label{prop::rpc_homeo}
The map $T$ is uniformly continuous with uniformly continuous inverse.  In particular, $T$ extends uniquely to a homeomorphism $\overline{T} \colon \overline{X} \to \overline{Y}$.
\end{proposition}

It is not hard to see that for $q \in \overline{X}$, $\overline{T}(q)$ is the same as the ROSt induced by the Bolthausen-Sznitman coalescent with overlap parameter $q$.  The most important consequence of Proposition \ref{prop::rpc_homeo} for us is that $\overline{Y}$ consists precisely of the continuous RPCs with Poisson-Dirichlet parameter $\zeta$.  Suppose that $(\xi,Q)$ is a finite level RPC.  We say that particles $\xi_i,\xi_j$ overlap at level $r$ provided that $l(i,j) = r$, $l(i,j)$ as defined in the previous section.  The proof of Proposition \ref{prop::rpc_homeo} depends on the following estimate that any of the largest $n$ particles overlap at a particular level.

\begin{lemma}
\label{lem::overlap_prob}
Let $q \in X$, $(\xi,Q)$ be an RPC with parameter $q$, and $0 = x_0 < x_1 < \cdots < x_{k} = \zeta < x_{k+1} = 1$ be the endpoints of the intervals on which $q$ is constant.  Let
\[ A_n^r = \bigcup_{1 \leq i < j \leq n} \{\xi_i,\xi_j \text{ overlap at level } r\}.\]
Then,
\[ |x_{r+1} - x_r| \leq \p(A_n^r) \leq C(n,\zeta)|x_{r+1} - x_r|,\]
where $C(n,\zeta) = \tfrac{1}{2\zeta}n(n-1)$.
\end{lemma}
\begin{proof}
Applying Proposition \ref{prop::rpc_construction} inductively, there exists independent processes $\eta^1,\ldots,\eta^{k-1}$ such that $\eta^l \sim PD(x_l/x_{l+1},0)$ for $1 \leq l \leq k-1$ and 
\[ Z = \p(A_2^r|\eta^1,\ldots, \eta^{k-1}) = \left( \prod_{l=r+1}^{k-1} \sum_{m=1}^\infty \eta_m^l(1-\eta_m^l) \right) \sum_{m=1}^\infty (\eta_m^r)^2\]
is the probability that $\xi_1,\xi_2$ overlap at level at $r$ conditional on $\eta^1,\ldots,\eta^{k-1}$.  For $n \geq 2$, observe
\[ Z \leq \p(A_n^r|\eta^1,\ldots,\eta^{k-1}) \leq \frac{n(n-1)}{2} Z.\]
If $\gamma \sim PD(x,0)$ and $n \geq 2$ then (Proposition A.1 of \cite{BS98}, Corollary 2.2(a) of \cite{RU87}),
\[ \E \sum_{m=1}^\infty \gamma_m^n = \frac{(n-1-x)(n-2-x) \cdots (1-x)}{(n-1)!}.\]
Hence if $1 \leq r \leq k-1$,
\begin{align*}
     \E Z
&= \E \left( \prod_{l=r+1}^{k-1} \sum_{m=1}^\infty \eta_m^l (1-\eta_m^l) \right) \sum_{m=1}^\infty (\eta_m^r)^2
  = \left(\prod_{l=r+1}^{k-1}  \left( 1- \E \sum_{m=1}^\infty (\eta_m^l)^2 \right) \right) \E \sum_{m=1}^\infty (\eta_m^r)^2\\
&= \left(\prod_{l=r+1}^{k-1} \frac{x_l}{x_{l+1}} \right) \left(1 - \frac{x_r}{x_{r+1}} \right)
  = \frac{1}{\zeta} (x_{r+1} - x_r).
\end{align*}
Similarly, if $r = 0$,
\begin{align*}
     \E Z
&= \left(\prod_{l=1}^{k-1} \frac{x_l}{x_{l+1}} \right)
 = \frac{1}{\zeta} (x_{1} - x_0).
\end{align*}
\end{proof}

\begin{proof}[Proof of Proposition \ref{prop::rpc_homeo}]
Let $\epsilon > 0$ be arbitrary and suppose $f,g \in X$ are such that $\| f - g\|_{L^1} \leq \epsilon^2$.  Let $(\xi,Q)$ and $(\eta,P)$ be random variables with the laws of $T(f)$ and $T(g)$.  We assume that $(\xi,Q)$ and $(\eta,P)$ are constructed by taking $\xi = \eta$ and $Q$ and $P$ generated from the same Bolthausen-Sznitman coalescent.  Hence the distance between $(\xi,Q)$ and $(\eta,P)$ is completely determined by $d(Q,P)$.  Take $n = [-\log_2 \epsilon ] + 1$ so that $\sum_{i \vee j \geq n+1} 2^{-i-j} \leq \epsilon$.  Then,
\[ d(Q,P) = \sum_{i,j=1}^\infty 2^{-i-j} |q_{ij} - p_{ij}| \leq \max_{1 \leq i,j \leq n} |q_{ij} - p_{ij}| + \epsilon.\]
Let $I_1 = [x_{i_1},x_{i_1+1}), \ldots, I_m = [x_{i_m},x_{i_m+1})$ be the disjoint intervals such that $E = \{|f-g| \geq \epsilon\} = \cup_{j=1}^m I_j$.  Using the notation of Lemma \ref{lem::overlap_prob},
\begin{align*}
     \p(\max_{1 \leq i,j \leq n} |q_{ij} - p_{ij}| \geq \epsilon)
&= \p(A_n^{i_1} \cup \cdots \cup A_n^{i_m})
 \leq C(n,\zeta) \sum_{j=1}^m |x_{i_j} - x_{i_j+1}|\\
&= \frac{n(n-1)}{2\zeta} |E|
  \leq \frac{n(n-1)}{2\zeta} \epsilon.
\end{align*}
Therefore,
\[ \p(d(Q,P) \geq 2\epsilon) \leq \frac{n(n-1)}{2\zeta} \epsilon.\]
The quantity on the right hand side clearly goes to zero as $\epsilon \to 0$.  This proves that $T$ is uniformly continuous.

Now assume $\p(d(Q,P) \geq \epsilon^2) \leq \epsilon$.  With $n = [-\tfrac{1}{2} \log_2 \epsilon]$, note
\[ d(Q,P) \geq 2^{-2n} \max_{1 \leq i,j \leq n} |q_{ij} - p_{ij}| \geq \epsilon \max_{1 \leq i,j \leq n} |q_{ij} - p_{ij}|.\]
Let $E$ and $I_{1},\ldots,I_m$ be as before.  As the collection $A_2^{i_1},\ldots,A_2^{i_m}$ is disjoint,
\[ \p(\max_{1 \leq i,j \leq n} |q_{ij} - p_{ij}| \geq \epsilon) = \p(A_n^{i_1} \cup \cdots \cup A_n^{i_m}) \geq \p(A_2^{i_1} \cup \cdots \cup A_2^{i_m}) = |E|.\]
Trivially, we have the bound,
\[ \| f -g \|_{L^1} \leq \epsilon + |E| \leq 2\epsilon.\]
Therefore $T^{-1}$ is uniformly continuous.
\end{proof}

\section{Approximation}
\label{sec::approx}

In this section we will show that every ROSt which is RQS for $\psi$ linear and non-constant, ergodic for the evolution, with $S_Q \subseteq (-1,1)$ (RQSE) can be approximated arbitrarily well by an RQSE ROSt with finite state space provided that $\overline{S_Q}$ has no limit points from below.  This means that if $q \in \overline{S_Q}$, $\sup\{ p \in \overline{S_Q} : p < q\} < q$.  The proof of Theorem 1.8 of \cite{AA07} implies that the restriction to $\psi(x) = x$ linear is without loss of generality.  Our strategy is to prove that if $g \colon [-1,1] \to [-1,1]$ is any non-decreasing right continuous step function with $g(x) \leq x$ for all $x \in [-1,1]$ and fixing $-1$, $0$, and $1$ then $(\xi,g(Q))$, here and hereafter $g(Q) = (g(q_{ij}))$ denotes entrywise application of $g$, is RQSE.  Establishing this fact consists of three main steps.  First, we will show that $S_Q$ has an $n$th largest element for every $n \leq |S_Q|$ that will necessarily be non-negative.  Second, by a scaling argument we can alter the first $n \leq |S_Q|$ elements of $S_Q$ maintaining positive-definiteness and RQSE so long as the order is preserved.  Third, it is possible to take a kind of limit that allows us to alter the entire state space.

Recall that a random measure $\nu$ on $\CH$, $\CH$ a separable Hilbert space, is said to direct $(\xi,Q)$ provided that conditional on $\nu$, $q_{ij} \stackrel{d}{=} (\phi_i,\phi_j) + \delta_{ij} (1-\|\phi_i\|^2)$ where $(\phi_i)$ is a $\nu$-iid sequence.  It is a consequence of Theorems 1.9 and 1.10 of \cite{AA07} that every RQSE ROSt is directed.  In the following two lemmas we show that the directing measure of such a ROSt is supported in a non-random sphere and that if $\overline{S_Q}$ does not have limit points from below, $\nu$ has countable support.

It is shown in \cite{A_GG07} that for each RQSE ROSt $(\xi,Q)$ and fixed $r, \lambda > 0$ there exists an extension of the probability measure 
\[ d\p_{r,\lambda} = d\p(\xi,Q) \times \prod_{t \geq 0} d \nu_{Q^{*r}}(\kappa(t)),\]
which is the law of $(\xi,Q)$ and its future increments under the $Q^{*r}$ evolution with $\psi(x) = \lambda x$, to $\Omega_{os} \times \prod_{t \in \Z} \R^\N$; the negative indices correspond to the past increments.  The construction follows by showing that the map $\Lambda \colon \Omega_{os} \times \prod_{t \geq -n} \R^\N \to \Omega_{os} \times \prod_{t \geq -n-1} \R^\N$ given by evolving $(\xi,Q)$ by the Gaussian sequence $(\kappa_i(0))$, reordering, and then shifting the index of the increments by $-1$ induces a consistent family of measures.  We abuse notation and write $\p_{r,\lambda}$ for the induced measure on $\Omega_{os} \times \prod_{t \in \Z} \R^\N$.  Then $\Lambda$ extends to a map $\Omega_{os} \times \prod_{t \in \Z} \R^\N \to \Omega_{os} \times \prod_{t \in \Z} \R^\N$ that preserves $\p_{r,\lambda}$ and $\p_{r,\lambda}$ is the unique extension such that the natural extension of $\Lambda$ has this property.  
%In particular, if there exists an equivalence relation $\sim$ on the particles that has the property that the $\sim$-equivalence classes are invariant under the evolution, then

We can think of the element $\kappa_i(-t)$ as the increment that particle $i$ received $t$ steps in the past.  It is proved in \cite{A_GG07} that the sequence of random sequences $\big( (\kappa_i(-t) : i \in \N) : t \in \N \big)$ is weakly exchangeable conditional on $(\xi,Q)$.  Hence modulo establishing a first moment bound (Proposition 2.4, \cite {A_GG07}), it follows from the strong law of large numbers that the limit
\[ v_i(r,\lambda) = \lim_{T \to \infty} \frac{1}{T} \sum_{t=1}^T \lambda \kappa_i(-t)\]
exists almost surely; $v_i(r,\lambda)$ is referred to as the past velocity of particle $i$ under the evolution \eqref{eqn::qs} with $\psi(x) = \lambda x$ and $(\kappa_i)$ with covariance $Q^{*r}$.  Proposition 2.5 of \cite{A_GG07} gives that $v_i(r,\lambda) = v(r,\lambda)$ a.s., so that the past velocity of the particles is a.s. the same.  It is shown in Lemma 2.7 of \cite{A_GG07} that $v(r,\lambda)$ admits the explicit formula,
\begin{equation}
\label{eqn::past_velocity}
 v(r,\lambda) = \lambda \int_{-1}^1 (1-q^r) dx(q),
\end{equation}
where $x(q) = \E \sum_{i,j} \xi_i \xi_j 1_{\{q_{ij} \leq q\}}$.  Observe that $(r,\lambda) \mapsto v(r,\lambda)$ completely determines $x(q)$.  Using this we can show that the directing measure $\nu$ must be supported on a non-random sphere.

\begin{lemma}
\label{lem::dir_sphere}
If $(\xi,Q)$ is RQSE then its directing measure $\nu$ is supported in a non-random sphere.
\end{lemma}
\begin{proof}
Conditional on $Q$ and $\nu$, we know that there exists a sequence $(\phi_i)$ contained in $\supp \nu$ such that $q_{ij} = (\phi_i,\phi_j) + \delta_{ij} (1-\|\phi_i\|^2)$ that has the property that $\overline{\{\phi_i\}} = \overline{\{\phi_i : i \geq N\}}$ for every $N \in \N$.  In other words, removing a finite number of elements from the sequence $(\phi_i)$ does not change the closure of the corresponding set.  Let $\epsilon > 0$ be arbitrary.  Suppose that $j > i$ is such that 
\[ \sup_{k > j} |q_{ik} - q_{jk} | = \sup_{k > j} |(\phi_i - \phi_j,\phi_k)| \leq \epsilon/4.\]
Since there exists $r > s > j$ such that $\|\phi_i - \phi_r\| \leq \epsilon/8$ and $\|\phi_j - \phi_s\| \leq \epsilon/8$, we have
\begin{align*}
          \| \phi_i - \phi_j \|^2
&\leq |(\phi_i - \phi_j, \phi_r)| + |(\phi_i - \phi_j, \phi_i - \phi_r)| + |(\phi_i - \phi_j, \phi_s)| + |(\phi_i - \phi_j, \phi_j - \phi_s)|\\
&\leq \epsilon.
\end{align*}
This implies that $\| \phi_i \|^2$ is determined completely by $q_{ij}$, $i \neq j$.  In particular, the length $\| \phi_i \|^2$ is intrinsic to the particle $\xi_i$.

Let $\epsilon > 0$ be arbitrary and let $\{\rho_k^2\}$ be an $\epsilon$-net of $[0,1]$.  Let $(\xi^k,Q^k)$ be the ROSt consisting of the particles $\xi_i$ such that $\rho_k^2$ is the largest point in the net smaller than $\| \phi_i \|^2$ normalized to have unit sum and $Q^k$ the corresponding overlap matrix (assuming non-empty).  Note that the event $(\xi^k,Q^k)$ is non-empty occurs with either probability $0$ or $1$ by ergodicity.  Furthermore, each non-empty $(\xi^k,Q^k)$ with probability $1$ has infinitely many particles as conditional on $\nu$ the $(\phi_i)$ are chosen iid.   Observe that $(\xi^k,Q^k)$ is RQSE as its correlated evolution is the evolution of $(\xi,Q)$ restricted to its own particles.  Let $v(r,\lambda)$ denote the velocity of $(\xi,Q)$ and $v^k(r,\lambda)$ the velocity of $(\xi^k,Q^k)$.  Then $v^k(r,\lambda) = v(r,\lambda)$ as the evolution of $(\xi^k,Q^k)$ is just the evolution of $(\xi,Q)$ restricted to $\xi^k$.  Hence it follows that $x^k(q) = \E \sum_{i,j} \xi_i^k \xi_j^k 1_{\{q_{ij}^k \leq q\}}$ is common among all of the non-empty ROSts $(\xi^k,Q^k)$ and equal to $x(q)$.  Therefore there can only be one $k_0$ such that $(\xi^{k_0},Q^{k_0})$ is non-empty.  This implies that $\supp \nu$ is contained in the annulus $\{ \phi \in \CH : \rho_{k_0}^2 \leq \| \phi\|^2 \leq \rho_{k_0}^2+ \epsilon \}$ and it follows from ergodicity that $\rho_{k_0}^2$ is constant.  Hence sending $\epsilon \to 0$ we see that $\nu$ is supported on a non-random sphere.
\end{proof}

\begin{lemma}
\label{lem::dir_det}
Suppose $(\xi,Q)$ is RQSE and directed by $\nu$.  Then the support of $\nu$ is countable.
\end{lemma}
\begin{proof}
By the previous lemma we know that $\nu$ is supported on a non-random sphere of radius $\rho^2$.  Conditional on $\nu$, let $(\phi_i)$ be a $\nu$-iid sequence so that $q_{ij} \stackrel{d}{=} (\phi_i,\phi_j)$, $i \neq j$.  Then $\phi_i = \phi_j$ if and only if $(\phi_i,\phi_j) = \rho^2$.  Let $q_2$ be the second largest element of $S_Q$ (which exists by our hypothesis on $Q$ and is deterministic by ergodicity).  Then if $\phi_i \neq \phi_j$, $(\phi_i,\phi_j) \leq q_2$.  Hence it follows that,
\[ \| \phi_i - \phi_j\|^2 = 2 \rho^2 - 2 (\phi_i,\phi_j) \geq 2 (\rho^2 - q_2) > 0\]
is strictly bounded from below so that the set of points $\{\phi_i : i \in \N\}$ in $\CH$ is discrete.  This can only happen if $\nu$ has countable support as $\CH$ is separable the sequence $(\phi_i)$ is iid.
\end{proof}

Theorems 1.9 and 1.10 of \cite{AA07} give that if $(\xi,Q)$ is RQSE then $\xi$ is distributed as a $PD(x,0)$ variable.  Hence if $\nu$ directs $(\xi,Q)$, conditional on $\nu$ we can view $(\xi,Q)$ as a marked $PD(x,0)$ variable $( (\xi_i,\phi_i) : i \in \N)$ with marking measure $\nu$.  One of the fundamental insights of \cite{AA07} (Lemma 2.7) is that the one step evolution of $(\xi,Q)$ under \eqref{eqn::qs} where $(\kappa_i)$ has covariance $Q^{*r}$ is also a marked $PD(x,0)$ random variable but with marking measure
\[ \nu'(d\phi) = \frac{e^{\lambda \kappa(T_r\phi)}}{N} \nu(d\phi).\]
Here, $N$ is a normalization so that $\nu'$ is a probability, $T_r \colon \CH \colon \to \CH$ is a map such that $(T_r \phi, T_r \psi) = (\phi,\psi)^r$, and $\kappa$ denotes an isomorphism of $\CH$ onto a Gaussian Hilbert space independent of $\xi$.  As $(\xi,Q)$ is RQS, the ROSt $(\xi,Q')$ with $Q'$ directed by $\nu'$ must have the same law as $(\xi,Q)$.  Hence we arrive at,

\begin{proposition}
\label{prop::rescale_direct}
Suppose $(\xi,Q)$ is a ROSt, $\xi \sim PD(x,0)$, directed by $\nu$.  Then $(\xi,Q)$ is RQS if and only if $\nu$ satisfies,
\begin{eqnarray}
\label{eqn::dir_eqn}
\nu(d\phi) \stackrel{d}{=} \frac{e^{\lambda \kappa(T_r \phi)}}{N} \nu(d\phi) \text{ for all } \lambda > 0, r \in \N
\end{eqnarray}
up to isometry of $\CH$.
In particular, if $S_Q \subseteq (-1,1)$, $0 < \alpha < (\sup S_Q)^{-1/2}$, and $f_\alpha \colon [-1,1] \to [-1,1]$ is the function
\[ f_\alpha(x) = \left\{ \begin{array}{cl} 1 \text{ for } x = 1\\ \alpha^2 x \text{ for } x \in (-1,1)\\ -1 \text{ for } x = -1 \end{array} \right.,\]
then $(\xi,f_\alpha(Q))$ is RQS for all $\lambda > 0$.
\end{proposition}
\begin{proof}
The first part is an immediate consequence of the above discussion; see also Theorem 4.2 of \cite{AA07}.  The second part follows from the scaling properties of \eqref{eqn::dir_eqn}.  Indeed, let $\nu'$ be given by scaling the support of $\nu$ by $\alpha$.  Then $\nu'$ satisfies,
\begin{align*}
     \nu'(d\phi)
&= \nu(d \alpha^{-1} \phi) 
   \stackrel{d}{=} \frac{e^{\lambda \kappa(T_r  \alpha^{-1} \phi)}}{N} \nu( d \alpha^{-1} \phi)
 = \frac{e^{ \lambda \alpha^{-r} \kappa(T_r \phi)}}{N} \nu'(d\phi),
\end{align*}
up to isometry.
\end{proof}

Arguing as in \cite{AA07}, by Lemmas \ref{lem::dir_sphere} and \ref{lem::dir_det} we can associate with $\nu = \sum_j \eta_j \delta_{\psi_j}$ a \emph{directing ROSt} $(\eta,P)$ with $P = (p_{ij})$, $p_{ij} = (\psi_i,\psi_j)/ \sup S_Q$.  That $\nu$ satisfies \eqref{eqn::dir_eqn} implies $(\eta,P)$ is RQSE since the evolution of the weights $\eta = (\eta_n)$ under the $Q^{*r}$ evolution of $(\xi,Q)$ corresponds to,
\[ (\eta,P) \mapsto \left( \left(\frac{\eta_i e^{\lambda' \kappa_i}}{\sum_{j} \eta_j e^{\lambda'  \kappa_j}} \right)_\downarrow, P^\pi \right)\]
where $(\kappa_i)$ is a Gaussian sequence with covariance $P^{*r}$, $\lambda' =  \lambda(\sup S_Q)^{r/2}$.  Indeed, the argument is similar to the proof of Proposition 4.3 of \cite{AA07}.  One does have to be careful here since in principle it could be that $-1 \in S_P$.  This, however, leads to a contradiction since if $p_{ij} = -1$ then the future and hence past increments of particle $i$ are exactly $-1$ times those of $j$.  In particular, if $v(r,\lambda)$ denotes the $(r,\lambda)$-velocity with respect to the evolution of $(\eta,P)$, then $v_i(1,1) = -v_j(1,1)$ so that $v_i(1,1) = v_j(1,1) = v(1,1) = 0$ and it is clear from \eqref{eqn::past_velocity} that $v(1,1) > 0$.  If $|S_Q| = \infty$ and $\overline{S_Q}$ does not have limit points from below, then $\overline{S_P} = \overline{\{ q/ \sup S_Q : q \in S_Q \setminus \{\sup S_Q\}\}}$ also has this property.  Hence this procedure can be iterated arbitrarily many times.  In particular, $S_Q$ has an $n$th largest element $a_n$ that is necessarily non-negative; note that $a_n$ can be expressed in terms of the radii of the support of the first $n$ directing measures.  The idea of the proof of the following proposition is to combine this with the second part of the previous proposition to show that we can change the top of $S_Q$ while preserving the property of RQSE.  Let $X$ be the set of non-decreasing, right-continuous step functions $f$ on $[-1,1]$ satisfying $f(x) \leq x$ for every $x \in [-1,1]$ and fixing $-1$, $0$, and $1$.

\begin{proposition}
\label{prop::change_state_space}
Suppose $(\xi,Q)$ is RQSE, $|S_Q| = \infty$, $\overline{S_Q}$ does not have limit points from below, and $f \in X$.  Then $(\xi,f(Q))$ is RQSE and indecomposable.
\end{proposition}
Note that it is not \emph{a priori} clear that the matrix $f(Q)$ should even be positive semidefinite; this, however, will be immediate from the proof.  In the following, we say that $(\xi,Q)$ is directed by $(\eta,P)$ with scaling factor $\rho$ provided $\xi$ is independent of $(\eta,P)$ and conditional on $(\eta,P)$, $q_{ij} \stackrel{d}{=} \rho p_{i^* j^*}$, $ i \neq j$, where $(i^*)$ is a sequence of $\eta$-iid random integers.
\begin{proof}
Let $a_n \geq 0$ be the $n$th largest element of $S_Q$.  Suppose that $g \colon [-1,1] \to [-1,1]$ is any right-continuous non-decreasing function such that $g(1) = 1$, $g(x) = x$ for all $x < a_N$ for some fixed $N \in \N$, and $g|[x,1]$ has finite range.    Our first claim is that $(\xi,g(Q))$ is RQSE.  To see this, let $\nu_1$ be the directing measure of $(\xi,Q)$ and $(\xi^1,Q^1)$ the associated directing ROSt.  Inductively let $\nu_k$ direct $(\xi^{k-1},Q^{k-1})$ for each $k \geq 2$.  Choose $\alpha_1,\ldots,\alpha_{N+1} \geq 0$ so that $\alpha_1 = g(a_1)$, $\alpha_1 \alpha_2 = g(a_2),\ldots, \alpha_1 \cdots \alpha_{N} = g(a_N)$ and $\alpha_1 \cdots \alpha_{N+1} = a_{N+1} < g(a_N)$.  Inductively set $(\xi^{k},P^k)$ to be the ROSt directed by $(\xi^{k+1},P^{k+1})$ with scaling factor $\alpha_{k+1}$ for $0 \leq k \leq N$ and $P^{N+1} = Q^{N+1}$.  Then $(\xi,P) = (\xi^0,P^0) = (\xi,g(Q))$, which proves the claim.

Let $\CA$ be the set of ROSts that are RQS,
\[ g_\alpha(x) = \left\{ \begin{array}{cl} f(x) \vee \alpha \text{ for } x \geq \alpha\\ x \text{ for } x < \alpha \end{array}\right. ,\]
and $A = \{ \alpha \in [-1,1] : (\xi,g_\alpha(Q)) \in \CA\}$.  The previous part gives us that $A \neq \emptyset$.  As $\CA$ is closed (Proposition \ref{prop::rqs_closed}) and, as $f \in X$, $\alpha \mapsto (\xi,g_\alpha(Q))$ is continuous hence $A$ is closed.  We just need to show that $A$ is open.  Suppose $\alpha \in A$ so that $(\xi,g_\alpha(Q)) \in \CA$.  The case when $|S_{g_\alpha(Q)}| < \infty$ is trivial.  Suppose $|S_{g_\alpha(Q)}| = \infty$.  As $(\xi,Q)$ is ergodic it follows that $(\xi,g_\alpha(Q))$ is as well and hence is RQSE.  Furthermore, $\overline{S_{g_\alpha(Q)}}$ does not have limit points from below.  Therefore $S_{g_\alpha(Q)}$ has an $n$th largest element $b_n \geq 0$ for every $n$.  As $g_\alpha$ takes on only a finite number of values for $x \geq \alpha$, there exists $N$ large enough so that $b_N < \alpha$.  Hence if $|\beta - \alpha| < \alpha - b_N$, $\beta \in A$.  Therefore $A$ is open, so that $A = [-1,1]$.

An immediate consequence of Lemma \ref{lem::dir_sphere} and an easy induction argument is that every RQSE ROSt with finite state space is indecomposable.  Hence the last part of the proposition follows.
\end{proof}

\noindent We can now prove Theorem \ref{thm::main_theorem}.

\begin{proof}[Proof of Theorem \ref{thm::main_theorem}]
Suppose $(\xi,Q)$ is RQSE.  By Proposition \ref{prop::change_state_space}, $(\xi,f(Q))$ is also RQSE and indecomposable with $f \in X$ provided $-1 < f(x) < 1$ for $x \in (-1,1)$.  Letting $\epsilon > 0$ be arbitrary we can choose such $f$ so that $d((\xi,Q),(\xi,f(Q))< \epsilon$.  By Theorem 1.8 of \cite{AA07} $(\xi,f(Q))$ is given by a finite level RPC.  By the discussion following Proposition \ref{prop::rpc_homeo}, we know that the closure of the set of RPCs with fixed Poisson-Dirichlet parameter is the set of continuous RPCs with the same parameter.  Therefore $(\xi,Q)$ is a continuous RPC.
\end{proof}

\section*{Acknowledgments}
I would like to thank Amir Dembo and Andrea Montanari for teaching a course in spin-glass models at Stanford University in the winter of 2008 that led to the this research problem as well as providing invaluable guidance in the preparation of this article.  I would also like to thank Louis-Pierre Arguin for helpful discussions and e-mail correspondence.

\bibliographystyle{plain.bst}	% (uses file "plain.bst")
\bibliography{qs_references}	

\end{document}